\documentclass[reqno,11pt]{amsart}
\usepackage{amsmath,amssymb,amsfonts,amsthm, amscd,indentfirst}
\usepackage{amsmath,latexsym,amssymb,amsmath,
	amscd,amsthm,amsxtra}

\usepackage[hyphens]{url}
\usepackage[colorlinks=true,linkcolor=blue,citecolor=blue,urlcolor=blue,hypertexnames=false,linktocpage]{hyperref}
\usepackage{bookmark}
\usepackage{amsmath,thmtools,mathtools}
\mathtoolsset{showonlyrefs=true}
\usepackage[T1]{fontenc}
\usepackage{newtxtext,newtxmath}

\newcounter{mparcnt}

\usepackage{fancyhdr}
\usepackage{esint}
\usepackage{enumerate}
\usepackage{xcolor}

\usepackage{pictexwd,dcpic}
\usepackage{graphicx}
\usepackage{caption}
\usepackage{graphicx}
\usepackage{caption}
\usepackage{epsfig,here}
\usepackage{subfigure,here}

\newtheorem{theorem}{Theorem}[section]
\newtheorem{proposition}[theorem]{Proposition}
\newtheorem{definition}[theorem]{Definition}
\newtheorem{remark}[theorem]{Remark}

\def\Om{\Omega}
\def\p{\partial}

\def\S{{\Sigma}}
\def\<{\langle}
\def\>{\rangle}

\providecommand{\abs}[1]{\lvert#1\rvert}

\newcommand{\mbR}{\mathbb{R}}
\newcommand{\mbS}{\mathbb{S}}

\newcommand{\mcC}{\mathcal{C}}

\newcommand{\mcH}{\mathcal{H}}

\newcommand{\mcK}{\mathcal{K}}

\newcommand{\mcW}{\mathcal{W}}

\newcommand{\rd}{{\rm d}}

\newcommand{\ra}{\rightarrow}

\newcommand{\eq}[1]{\begin{equation}\begin{alignedat}{2} #1 \end{alignedat}\end{equation}}

\numberwithin{equation} {section}

\begin{document}
	
	\title[Heintze-Karcher inequality]{Heintze-Karcher inequality for anisotropic free boundary hypersurfaces in convex domains}

\author[Jia]{Xiaohan Jia}
	\address[X.J]{School of Mathematical Sciences\\
		Xiamen University\\
		361005\\ Xiamen\\ P.R. China}
	\email{jiaxiaohan@xmu.edu.cn}
	
		\author[Wang]{Guofang Wang}
	\address[G.W]{Mathematisches Institut\\
		Universit\"at Freiburg\\
	Ernst-Zermelo-Str.1\\
		79104\\
  \newline\indent Freiburg\\ Germany}
	\email{guofang.wang@math.uni-freiburg.de}
 
\author[Xia]{Chao Xia}
	\address[C.X]{School of Mathematical Sciences\\
		Xiamen University\\
		361005\\ Xiamen\\ P.R. China}
	\email{chaoxia@xmu.edu.cn}
	\author[Zhang]{Xuwen Zhang}
	\address[X.Z]{School of Mathematical Sciences\\
		Xiamen University\\
		361005\\ Xiamen\\ P.R. China
\newline\indent Institut F\"ur Mathematik\\ Goethe-Universit\"at\\ Robert-Mayer-Str.10\\60325\\Frankfurt\\
\newline\indent Germany}
	\email{zhang@math.uni-frankfurt.de}
	
	\thanks{C.X. is supported by the National Natural Science Foundation of China (Grant No. 12271449, 12126102). X.Z. is supported by CSC (No. 202206310078) and the NanQiang excellent Ph.D. student project of Xiamen University.}
	
\begin{abstract}
In this paper, we prove an optimal Heintze-Karcher-type inequality for anisotropic free boundary hypersurfaces in general convex domains. The equality is achieved for anisotropic free boundary Wulff shapes in a convex cone. As applications, we prove various Alexandrov-type theorems.
		
\
		
\noindent {\bf MSC 2020:} 53C45, 53A10, 53C42.\\
{\bf Keywords:} Heintze-Karcher's inequality,  Constant mean curvature, Free boundary surface, Capillary surface, Convex cone.\\
		
\end{abstract}
	
	\maketitle
	
	\medskip
\section{Introduction}
The Heintze-Karcher inequality states that for a bounded domain $\Om$ in $\mbR^{n+1}$ with smooth and mean convex boundary $\S=\p\Omega$, it holds that 
\eq{\label{hk}
\int_\S\frac{1}{H}\rd A
\ge\frac{n+1}{n}\abs{\Omega},
}
and equality in \eqref{hk} holds if and only if $\S$ is a geodesic sphere. Here $H$ is the mean curvature of $\S$ and $\S$ is said to be mean convex if $H>0$. It was first proved by Heintze-Karcher \cite{HK78} and in the present form \eqref{hk} by Ros \cite{Ros87}.
A combination of the Heintze-Karcher inequality \eqref{hk} and the Minkowski-Hsiung formula yields Alexandrov's theorem for embedded closed constant mean  curvature (CMC) hypersurfaces, namely any closed embedded  CMC hypersurface is a sphere, first proved by Alexandrov \cite{Alex62} via moving plane method and by Reilly \cite{Reilly77} and Ros \cite{Ros87} via integral method.
For the history of the study of Heintze-Karcher's inequality and Alexandrov's theorem, we refer to \cite{JWXZ22} and the references therein.
An anisotropic version of the Heintze-Karcher inequality has been proved by He-Li-Ma-Ge \cite{HLMG09}, which   leads  to an Alexandrov-type theorem for embedded closed constant anisotropic mean curvature hypersurfaces. Heintze-Karcher-type inequalities in space forms and more generally, in warped product manifolds, has been proved by Brendle \cite{Brendle13} (see also \cite{LX19}).

Recently, we have studied the Heintze-Karcher-type inequality for hypersurfaces with boundary in the half-space $\mbR^{n+1}_+=\{x\in\mbR^{n+1}| \<x, E_{n+1}\>>0\}$. Precisely, we prove the following result.
\begin{theorem}[\cite{JWXZ22, JWXZ23}] \label{thm1.1} 

Let $\S\subset\overline{\mbR^{n+1}_+}$ be an embedded compact $C^2$-hypersurface with boundary $\p\S$ which intersects $\p\mbR^{n+1}_+$ transversally such that  $\<\nu_F(x), -E_{n+1}\>\le 0$ for any $x\in \p\S$.
Let $\Om$ denote the domain enclosed by $\Sigma$ and $\p \mbR^{n+1}_+$.
Assume the anisotropic mean curvature $H^F$ of $\S$ is positive. Then we have
\eq{\label{ineq-HK-JWXZ}
\int_\S\frac{F(\nu)}{H^F}\rd A
\geq\frac{n+1}{n}\abs{\Om}.
}
Moreover, equality in \eqref{ineq-HK-JWXZ} holds if and only if $\S$ is an anisotropic free boundary Wulff cap.
\end{theorem}
A hypersurface  $\S$, which intersects $\p\mbR^{n+1}_+$ transversally, is called anisotropic free boundary if  $\<\nu_F(x), -E_{n+1}\>= 0$ for any $x\in \p\S$, where $\nu_F$ is the anisotropic normal of $\S$. A Wulff cap is part of a Wulff shape.
Regarding the anisotropy $F$ and anisotropic quantities with a sub- or superscription $F$, we use the notations in \cite{JWXZ23}.
A brief overview is also provided in Section \ref{Sec-2}.

Using Theorem \ref{thm1.1}, we have proved an Alexandrov-type theorem for anisotropic free boundary or anisotropic capillary hypersurfaces in $\mbR^{n+1}_+$. This Alexandrov-type theorem for isotropic capillary hypersurfaces in $\mbR^{n+1}_+$ was first proved by Wente \cite{Wente80} via moving plane method.

In this paper, we continue the study of Heintze-Karcher-type inequality in a more general setting, in particular, for hypersurfaces with boundary in general convex domains.

Let $\mcK$ be the class of all convex domains in $\mbR^{n+1}$.
We emphasize that in this paper for each $K \in \mcK$, $\p K$ is not assumed to be $C^1$, in other words, $\p K$ may have singularity. 
In order not to make situations too complicated, we restrict us to consider a subset of $\mcK$, denoted by $\mcK_P$, which contains  convex domains with milder singularities, which we will describe. 

We first define the class of convex polytopes with non-empty interior by $\mathcal P$. Each element $P\in {\mathcal P}$ can be  expressed 
\eq{
P=\bigcap_{j\in J} \{u_j \le 0\},
}
where $J$ is a finite index set and $u_j$ $(j\in J)$ are affine linear functions on $\mbR^{n+1}$, i.e., $P$ is determined by a family of inequalities $ u_j\le 0$ $(j\in J)$. 
The family is called {\it irredundant}
if  $P\not = \bigcap_{j\in J, j\not =l} \{u_j \le 0\} $
for any $l\in J$.
We can and will always assume that $P$ is irredundant.
A face of $P$ is defined by $P\cap\bigcap\limits_{i\in I\subseteq J} \{u_i=0\}$.
The faces of dimensions $0, 1,  (n+1)- 2$ and $ (n+1)-1$ are called \textit{vertices}, \textit{edges}, \textit{ridges}, and \textit{facets}, respectively,
see e.g., \cite{Gunter95}.
In this paper we will call ridges
{\it edges}, since our main interested case is $n+1=3$, and in this case the ridges and edges are actually the same.
The (relative) interior of  facets are the smooth part of $\p P$ and the other belongs to
the singular part of $\p P$. Each facet  of $ P$ can be expressed by
$\{u_i=0\} \cap P\eqqcolon f_i$ and each edge (in fact ridge) by
$f_i\cap f_j$ for some $i\neq j\in J$. Let $f_{ij}^0$ denote the relative interior of $f_i\cap f_j$.

Next, we introduce the subclass $ \mcK_P\subset\mcK$, whose elements are called {\it polytope-type convex domains}. For each $ K\in  \mcK_P$, there exists a bi-Lipschitz homeomorphism  $\Psi :P \to  K$ 
from a convex polytope $P= \bigcap _{j\in J} \{u_j\le 0\}$ such that
$\Psi_{|f_j}:  f_j \to \Psi (f_j)$ is smooth 
for any $j\in J$. 
Set $ F_j \coloneqq\Psi (f_j)$ ($j\in J$).
We denote by ${\rm Reg}(\p K)$ to be the set of regular points  in $\p K$ and ${\rm Sing}(\p K)=\p K\setminus {\rm Reg}(\p K)$.   It is clear that
${\rm Reg}(\p K)$ is the union of  the sets of relative interior of $ F_j$. We consider a subset of the singular part of $\p  K$ defined by
\eq{{\rm Sing}_0(\p  K)
\coloneqq\bigcup_{i,j \in J}\Psi (f_{ij}^0).
}
By definition, for any point $x\in{\rm Sing}_0(\p  K) $, there exists two smooth pieces
$ F_i$ and $ F_j$ such that $x $ belongs to the (relative) interior of $ F_i\cap  F_j$. 
For $x\in {\rm Reg}(\p K)$, we denote by $\bar N(x)$ to be the outward unit normal to $\p K$ at $x$. For $x\in{\rm Sing}_0(\p  K)$, we denote by $\bar N_{ij}^1(x)$ and $\bar N_{ij}^2(x)$ be the outward unit normals to $ F_i$ and $ F_j$ at $x$ respectively, if $x\in {\rm int}( F_i\cap  F_j)$.

We shall point out that  the study of these polytope-type domains has become modern interest in geometry.
For example, in \cite{Gromov14,Li21}, polyhedron comparison theorems for 3-manifolds of nonnegative scalar curvature that are modelled from Euclidean cube-type, cone-type, and prism-type polyhedron are proved.
In \cite{EL22}, Allard-type regularity theorem and a partial regularity theorem  are proved to hold for (locally) polyhedral cone domains.
In our recent work \cite{JWXZ22}, we proved non-existence results and  Alexandrov-type theorems for capillary hypersurfaces, when the ambient container is a wedge, namely, interesection of two intesecting half-spaces.
Note that by the definition above, the blow-up limit (tangent cone) of $ K\subset \mathcal{K}_P$ at any $x\in{\rm Sing}_0(\p K)$ is exactly given by a wedge, which becomes the major motivation for us to consider Theorem \ref{Thm-HK-Divisor} below.

\subsection{Main results}
Our first result in this paper is an anisotropic Heintze-Karcher inequality for hypersurfaces with boundary in convex domains, which extends the results for the half-space case in \cite{JWXZ23}.
\begin{theorem}\label{Thm-HK}
Let  $ K\in\mcK$ be a convex domain with boundary $\p  K$ and $\S\subset\overline{ K}$  an embedded compact $C^2$-hypersurface with boundary $\p\S\subset{\rm Reg}(\p K)$ intersecting $\p K$ transversally such that for $x\in\p\S$,
\eq{\label{condi-anisotropic-freebdry}
\langle\nu_F(x),\bar N(x)\rangle\leq0.
} Let $\Om$ denote the domain enclosed by $\Sigma$ and $\p K$.
Assume the anisotropic mean curvature $H^F$ of $\S$ is positive. Then we have
\eq{\label{ineq-HK}
\int_\S\frac{F(\nu)}{H^F}\rd A
\geq\frac{n+1}{n}\abs{\Om}.
}
{
Moreover, equality in \eqref{ineq-HK} holds if and only if $\S$ is an anisotropic free boundary Wulff cap  and
$\p K\cap\overline{\Omega}$ is a part of the boundary of the convex cone $\mcC$ with a smooth sector that is determined by $\S$.
}
\end{theorem}
Here a convex cone $\mcC=\mcC_\omega\subset \mbR^{n+1}$ is given by 
$$\mcC=x_0+\{tx\in \mbR^{n+1}, t\in [0,\infty), x\in \omega\},
$$
where $\omega\subset\mbS^n$ is a  spherical convex domain, which we call the \textit{sector of $\mcC$}, { and $x_0\in\mbR^n$ is the vertex of $\mcC$}. 
We say that $\omega$ is smooth if $\p \omega$ is smooth in $\mbS^n$.

In the above theorem, the assumption $\p\S\subset{\rm Reg}(\p K)$, which means that $\p\S$ lies in the regular parts of $\p K$, is a technical condition for our proof.
The case when $\p\S$ contains singular points of $\p K$ becomes much more complicated,
as one could easily observe from the proof of Theorem \ref{Thm-HK} that the boundary behavior of $\S$ is essential to show the surjectivity of the flow that produces parallel hypersurfaces.
Motivated by the wedge case tackled in \cite{JWXZ22}, our next result says that in the isotropic case,  the Heintze-Karcher inequality remains true when $\p\S$ has non-empty intersection with ${\rm Sing}_0(\p K)$, provided that $ K\in \mcK_P$.


\begin{theorem}\label{Thm-HK-Divisor}
Let  $ K\in \mcK_P$ be a polytope-type convex domain and $\S\subset\overline{ K}$ be an embedded,
compact, mean convex $C^2$-hypersurface with boundary $\p\S$.
Assume that $\p\S\subset{\rm Reg}(\p K)\cup{\rm Sing}_0(\p K)$ intersects $\p K$ transversally such that
\begin{enumerate}
\item $\langle\nu(x),\bar N(x)\rangle\leq0$ on $\p\S\cap{\rm Reg}(\p K)$,
\item
For $x\in\p\S\cap{\rm Sing}_0(\p K)$, 
\eq{\label{condi-freebdry-Divisor}
\langle\nu(x),\bar N_{ij}^{\alpha}(x)\rangle\leq0, \text{ }\alpha=1,2,
}
when $x\in {\rm int}( F_i\cap  F_j)$.
\end{enumerate}
Then we have
\eq{\label{ineq-HK-Divisor}
\int_\S\frac{1}{H}\rd A
\geq\frac{n+1}{n}\abs{\Om}.
}
Moreover, equality in \eqref{ineq-HK-Divisor} holds if and only if
$\S$ is a  free boundary spherical cap and
$\p K\cap\overline{\Omega}$ is a part of the boundary of the convex cone $\mcC$
whose sector is determined by $\S$.
\end{theorem}

As consequences of Theorem \ref{Thm-HK} and Theorem \ref{Thm-HK-Divisor},
we can prove the Alexandrov-type theorems for (anisotropic) free boundary hypersurfaces in convex cones with possibly singular sector. 

\begin{theorem}\label{Thm-Alex}
Let $\mcC\subset\mbR^{n+1}$ be a convex cone {with vertex at the origin}. 
Let $\S\subset\overline{\mcC}$ be a compact, embedded, anisotropic free boundary $C^2$-hypersurface with constant anisotropic mean curvature, such that $\p\S\subset{\rm Reg}(\p\mcC)$.
Then $\S$ must be an anisotropic free boundary Wulff cap. 
Moreover,  there are two possible scenarios: 
\begin{enumerate}
\item $\S$ is a Wulff cap centered at the origin, and in this case the sector of $\mcC$ is smooth;
\item $\p\S$ lies on a flat portion of $\p\mcC$.
\end{enumerate}

\end{theorem}

\begin{theorem}\label{Thm-Alex-Divisor}
Let $\mcC\in \mcK_P$ be a polytope-type convex cone {with vertex at the origin}.
Let $\S\subset\overline{\mcC}$ be a compact, embedded, free boundary $C^2$-hypersurface with constant mean curvature, such that  $\p\S\subset{\rm Reg}(\p\mcC)\cup{\rm Sing}_0(\p\mcC)$. 
Then $\S$ must be a free boundary spherical cap.
Moreover, there are two possible scenarios: 
\begin{enumerate}
\item $\S$ is a free boundary spherical cap centered at the origin;
\item $\p\S$ lies on a flat portion or a wedge portion of $\p\mcC$.
\end{enumerate}

\end{theorem}
\begin{remark}
\normalfont
In the case that $\mcC\subset\mbR^{n+1}$ is a convex cone with smooth sector,
Theorem \ref{Thm-Alex-Divisor} has been proved by Choe-Park \cite{CP11} via Reilly's formula (see also Pacella-Tralli \cite{PT20}).
In the case that $\mcC$ is a wedge, Theorem \ref{Thm-Alex-Divisor} has been proved by the authors in \cite{JWXZ22}. 

\end{remark}

By virtue of the above results, 
it is natural to raise the following question:
are Theorem \ref{Thm-HK-Divisor} and Theorem \ref{Thm-Alex-Divisor} true in the anisotropic case?

\subsection{Further discussions: on various Heintze-Karcher-type inequalities}

Theorem \ref{Thm-HK} includes, or in a better way to say, unifies many previous Heintze-Karcher-type inequalities.
First of all, it is obvious that all closed surfaces can be viewed as a free boundary hypersurface  (in fact, with empty boundary), and hence Theorem \ref{Thm-HK} can be applied to the closed case.

\subsubsection{Capillary hypersurfaces in the half-space} Let $ K$ be the upper half-space $\mathbb{R}^{n+1}_+$ with boundary $\partial \mathbb{R}^{n+1}_+$. Let $\Sigma$ be a hypersurface in $\mathbb{R}^{n+1}_+$ supported on $\partial\mathbb{R}^{n+1}_+$.
If $\Sigma$ is a usual capillary hypersurface, i.e., it intersects the support, $\p\mathbb{R}^{n+1}_+$, at a constant angle $\theta$, we can view it as an anisotropic free boundary  hypersurface with respect to a special $F$ intentionally defined as
\[
F(\xi)
=\abs{\xi}-\cos\theta \langle \xi,E_{n+1}\rangle,
\]
where $E_{n+1}$ is the $(n+1)$-$th$ unit vector. In fact, it is easy to see that in this case $\nu_F=\nu-\cos \theta E_{n+1},$ and hence
$
\langle \nu_F, -E_{n+1}\rangle=\langle \nu, -E_{n+1}\rangle +\cos\theta=0.$
Another interesting observation is that for this $F$ the anisotropic mean curvature equals to the usual one, i.e.,
$H_F=H.$
Hence applying Theorem \ref{Thm-HK} we obtain
\[
\int_\S \frac {1-\cos\theta \langle \nu, E_{n+1}\rangle} H \rd A\ge \frac 1 {n+1} |\Omega|,
\]
which was proved in our previous paper \cite{JWXZ22}. This observation has been given in \cite{LXZ23}.

\subsubsection{Free boundary hypersurfaces in the unit ball}
Let $ K$ be the unit  ball $\mathbb{B}^{n+1}$. If $\S$ is a free boundary hypersurface in $\overline{\mathbb{B}^{n+1}}$ supported on $\mathbb{S}^n$. Then from Theorem \ref{Thm-HK} we have a geometric inequality
\eq{\label{eq3z}
\int_\Sigma \frac 1 H \rd A 
\ge \frac {n+1} n |\Omega|.
}
\eqref{eq3z} could not be improved in this form, but is not an optimal inequality, since equality never holds again by Theorem \ref{Thm-HK}. An optimal version of Heintze-Karcher inequality with free boundary was established in \cite[Theorem 5.2]{WX19}.
It also leaves an interesting question:
Whether one can establish an optimal Heintze-Karcher inequality for capillary hypersurfaces in $\mathbb{B}^{n+1}$?

\

The rest of this paper is organized as follows. In Section \ref{Sec-2}, we review the notions involving anisotropy and also  anisotropic free boundary or capillary hypersufaces. In Section \ref{Sec-3}, we prove the Heintze-Karcher-type inequalities in Theorems \ref{Thm-HK} and \ref{Thm-HK-Divisor}.
In Section \ref{Sec-4}, we prove Alexandrov-type theorems in Theorems \ref{Thm-Alex} and \ref{Thm-Alex-Divisor}.

\

\section{Preliminaries}\label{Sec-2}
Let $F: \mathbb{S}^n\to \mathbb{R}_+$ be a $C^2$ positive function on $\mathbb{S}^n$ such that 
\begin{equation}\label{convexity condition}
(\nabla^2 F+F \sigma)>0,
\end{equation} where $\sigma$ is the canonical metric on $\mathbb{S}^n$ and $\nabla^2$ is the Hessian on $\mathbb{S}^n$.
One can extend $F$ as a positively
homogeoneous one  function on $\mathbb{R}^{n+1}$, namely, $F(\xi)=|\xi|F(\frac{\xi}{|\xi|})$ for $\xi\neq 0$ and $F(0)=0$. \eqref{convexity condition} is then equivalent to say that $\frac12F^2$ is uniformly convex in the sense $$D^2(\frac12F^2)(\xi)>0\hbox{ for }\xi\neq 0.$$
Here $D^2$ is the Euclidean Hessian.
We denote
\eq{
A_F=\nabla^2 F+F\sigma.
}
Let $\Phi:\mathbb{S}^n\to\mathbb{R}^{n+1}$ be the Cahn-Hoffman map associated with $F$  given by
\begin{eqnarray*}
 \Phi(x)
=\nabla F(x)+F(x)x,
\end{eqnarray*}
where $\nabla$ denote the gradient on $\mbS^n$. One easily sees that  the image $\mathcal{W}_F=\Phi(\mbS^n)$  is a strictly convex, closed hypersurface in $\mbR^{n+1}$, which is called the unit Wulff shape with respect to $F$.

Let $F^o:\mbR^{n+1}\ra\mbR$ be defined by
\eq{
    F^o(x)=\sup\left\{\frac{\langle x,z\rangle}{F(z)}\Big| \, z\in\mbS^n\right\},
}
where $\langle\cdot,\cdot\rangle$
denotes the standard Euclidean inner product. It is  a well-known fact (see e.g., \cite{HLMG09}) that the unit
Wulff shape $\mcW_F$ can be interpreted by $F^o$ as
\begin{align*}
\mcW_F=\{x\in\mbR^{n+1}|F^o(x)=1\}. 
\end{align*}
A Wulff shape of radius $r$ centered at $x_0\in \mbR^{n+1}$ is given by \begin{align*}
\mcW_{r_0}(x_0)=\{x\in\mbR^{n+1}|F^o(x-x_0)=r_0\}.
\end{align*}
{We call the part of a Wulff shape truncated by a half-space {\it a Wulff cap}. }

Let $\S\subset\overline{ K}$ be an embedded, $C^2$-hypersurface with $\p\S\subset\p K$, which encloses a bounded domain $\Om$.
Let $\nu$ be the unit normal of $\S$ pointing outward $\Om$. 
The \textit{anisotropic normal} of $\S$ is given by
\eq{
\nu_F=\Phi(\nu)=\nabla F(\nu)+F(\nu)\nu,
}
and the \textit{anisotropic principal curvatures} $\{\kappa_i^F\}_{i=1}^n$ of $\S$ are given by the eigenvalues of the \textit{anisotropic Weingarten map} $$\rd\nu_F=A_F(\nu)\circ\rd\nu: T_p\S\to T_p\S.$$ The eigenvalues are real since $(A_F)$ is positive definite and symmetric.
Let $H^F=\sum_{i=1}^n\kappa_i^F$ denote the \textit{anisotropic mean curvature} of $\S$. 
It is easy to check that the anisotropic principal curvatures of $\mcW_r(x_0)$ are $\frac1r$, since 
\begin{align}\label{normal}
    \nu_F(x)=\frac{x-x_0}{r},  \quad\hbox{ on } \mcW_r(x_0).
\end{align}

In the same spirit of \cite{JWXZ23}, we introduce
\begin{definition}[Anisotropic free boundary  hypersurface]\label{Defn-free-bdry-Aniso}
\normalfont
$\S\subset\overline{ K}$ is a \textit{anisotropic free boundary  hypersurface w.r.t. $F$} in $ K$  if for every $x\in\p\S$, there holds
\eq{\label{defn-anisotropic-freebdry-1}
\langle\nu_F(x),\bar N(x)\rangle=0,
}
where $\bar N$ denotes the outer unit normal field on $\p K$. 
When $F\equiv 1$, we say $\S$ is a free boundary hypersurface.\end{definition}
Moreover, we may define the \textit{anisotropic co-normal} as (see \cite[(3.6)]{GX23}, also \cite{Rosales23})
\eq{
\mu_F
\coloneqq F(\nu)\mu-\langle\nu_F,\mu\rangle\nu,
}
where $\mu$ denotes the outer unit co-normal of $\p\S$ in $\S$. 
A direct computation yields that
\begin{proposition}
Along $\p\S$, one has
\eq{\label{eq-mu_F-nu_F-1}
\langle\mu_F,\nu_F\rangle=0,
}
and
\eq{\label{eq-mu_F-nu_F-23}
\langle\mu_F,\bar N\rangle
=&\langle\nu_F,\bar\nu\rangle,\qquad
\langle\mu_F,\bar\nu\rangle
=&-\langle\nu_F,\bar N\rangle,
}where  $\bar\nu$ is the outer unit co-normal of $\p\S$ in $\p K$.
\end{proposition}
\begin{proof}
\eqref{eq-mu_F-nu_F-1} follows readily from the fact $\langle\nu_F,\nu\rangle=F(\nu)$ 
and the definition of $\mu_F$.
\eqref{eq-mu_F-nu_F-23} follows similarly after noticing that
\eq{\label{eq2.6}
\bar\nu
&=-\langle\nu,\bar N\rangle\mu+\langle\mu,\bar N\rangle\nu,\\
\bar N
&=\langle\nu,\bar \nu\rangle\mu-\langle\mu,\bar \nu\rangle\nu.
}
\end{proof}
\eqref{eq2.6} implies that  $\{\nu,\mu\}$ and $\{\bar\nu,\bar N\}$ span the same $2$-plane.
It follows from \eqref{eq-mu_F-nu_F-23} and \eqref{defn-anisotropic-freebdry-1} that
\eq{
\langle\mu_F,\bar\nu\rangle
=-\langle\nu_F,\bar N\rangle=0,
}
from this and the fact that $\mu_F$ is a linear combination of $\mu,\nu$, we deduce that for an anisotropic free boundary  hypersurface in $ K$, there holds: for every $x\in\p\S$,
\eq{\label{condi-aniso-conormal}
\mu_F(x)\parallel\bar N(x).
} 

Figure \ref{Fig-1} indicates an example of the anisotropic free boundary  hypersurface. In general $\nu_F$ may not belong to the $2$-plane spanned by $\nu,\mu$, which makes a big difference to the isotropic case, i.e., $F(\xi)=|\xi|$.
However, $\mu_F$ always
lies in the 2-plane by the  definition and is in fact
parallel to $\bar N$, if $\Sigma$ is an anisotropic free boundary  hypersurface. In general it is natural to introduce 

\begin{definition}
[Anisotropic capillary hypersurface]
\normalfont
Let $\Sigma$ be a hypersurface supported by $S$ as above.
We call $\Sigma$ {\it an anisotropic capillary hypersurface with contact angle $\theta$ with respect to $F$} if it satisfies
\eq{
\frac{\mu_F}{\abs{\mu_F}}
=\sin\theta \bar N +\cos\theta \bar \nu.
}
\end{definition}
If $F(\xi)=|\xi|$, 
i.e., if we are in the isotropic case, the definition is the same as the usual definition.
An anisotropic free boundary  hypersurface is an anisotropic capillary hypersurface with contact angle $\pi/2$. 


\begin{figure}[H]
	\centering
	\includegraphics[width=12cm]{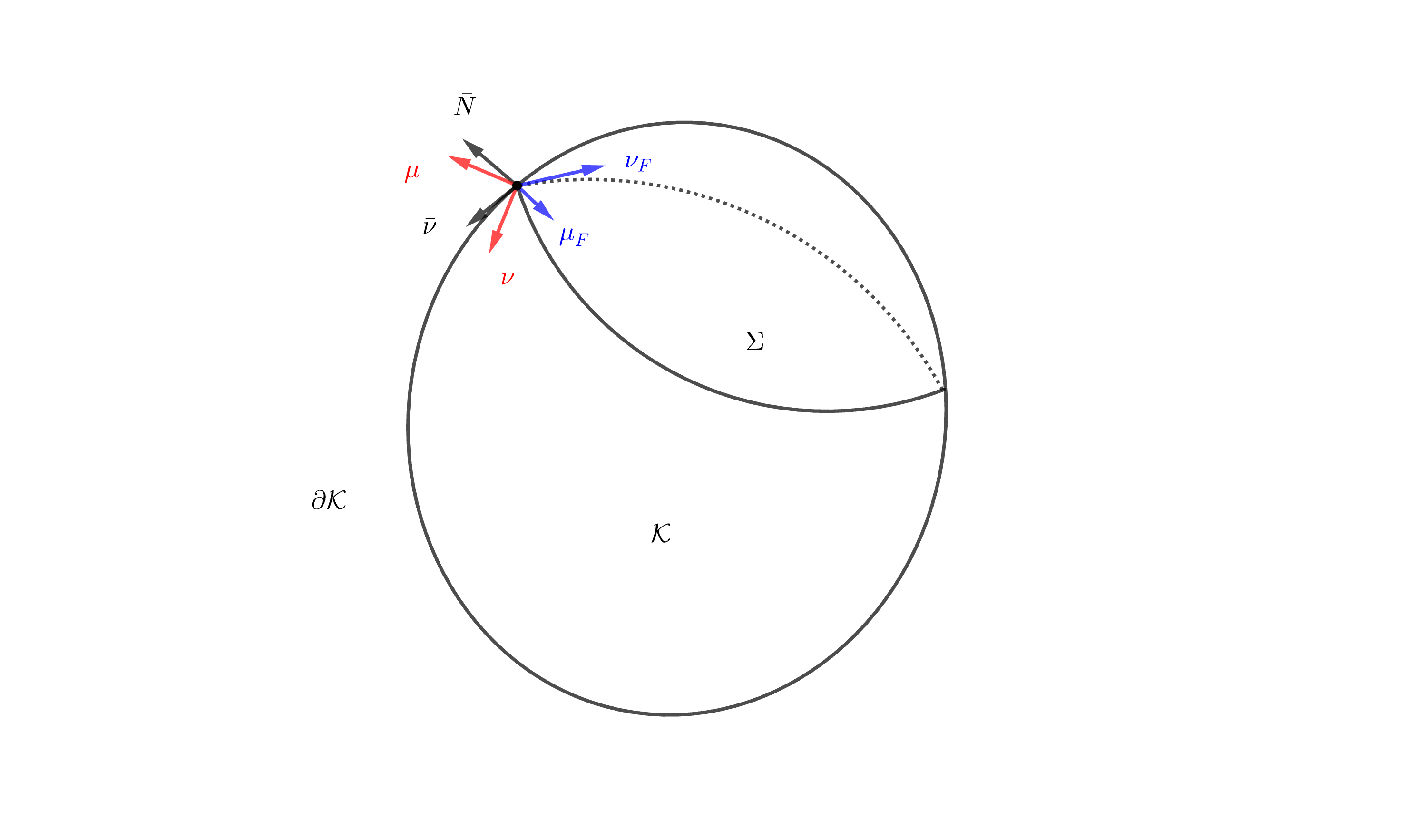}
	\caption{anisotropic free boundary  hypersurface.}
	\label{Fig-1}
\end{figure}

\section{Heintze-Karcher-Type Inequalities}\label{Sec-3}
\subsection{Proof of Theorem \ref{Thm-HK}}
In order to prove the main result, a key ingredient we shall exploit is the following  monotonicity principle in the anisotropic sense.
\begin{proposition}[{\cite[Proposition 3.3]{JWXZ23}}]\label{Prop-angle-compare}
Let $x,z\in\mbS^n$ be two distinct points and $y\in\mbS^n$ lies in a length-minimizing geodesic joining $x$ and $z$ in $\mbS^n$,
then we have \eq{
\langle\Phi(x),z\rangle\leq\langle\Phi(y),z\rangle.
}
Equality holds if and only if $x=y$.
\end{proposition}
\begin{proof}
For the convenience of the reader we provide its proof.

We denote  $d_0=d_{\mbS^n}(x,z)$ and $d_1=d_{\mbS^n}(x,y)$, where $d_{\mbS^n}$ denotes the intrinsic distance on $\mbS^n$.
If $y\neq x$, clearly $0<d_1\leq d_0$.
Let $\gamma:[0,d_0]\to \mbS^n$ be the arc-length parameterized geodesic with $\gamma(0)=x$, $\gamma(d_0)=z$.  Consider the following function
\eq{
f=\<\Phi(\gamma(t)),z\>, \quad t\in [0,d_0].
}
We have
\eq{\label{eq-prop}
  & \<\Phi(y),z\>-\<\Phi(x),z\>=f(d_1)-f(0)\nonumber\\ =& \int_0^{d_1} \left<\frac{\rd}{\rd t}\Phi(\gamma(t)),z\right>\rd t=\int_0^{d_1}\<D_{\dot{\gamma}(t)}\Phi(\gamma(t)), z\>\rd t,
}
where $D$ is the Euclidean covariant derivative.
Since $\gamma$ is length-minimizing, it is easy to see that  \eq{\<\dot{\gamma}(t), z\>\ge 0, \quad \forall t\in (0, d_0).
}
Thus $z$ can be expressed as $z=\sin s\dot{\gamma}(t)+\cos s\gamma(t)$ with some $s\in(0,\pi)$. It follows that
\eq{
  \<D_{\dot{\gamma}(t)}\Phi(\gamma(t)), z\>=\sin s (\nabla^2 F+F I)(\dot{\gamma}(t),\dot{\gamma}(t)),
}
Since  $(\nabla^2 F+F I)>0$,
we get $\<D_{\dot{\gamma}(t)}\Phi(\gamma(t)), z\>>0$ for any $t\in(0,d_1)$.
This fact, together with \eqref{eq-prop}, leads to the assertion.
\end{proof}


\begin{proof}[Proof of Theorem \ref{Thm-HK}]
For any $x\in\S$, let $e_i^F(x)$ be the corresponding anisotropic principal vector of $\S$ at $x$ with respect to the anisotropic principal curvature $\kappa_i^F(x)$, such that $\abs{e_1^F\wedge\ldots e_n^F}=1$.
Since $\S$ is strictly anisotropic mean convex,
\eq{
\max_i{\kappa^F_i(x)}\ge \frac1n H^F(x)>0, \hbox{ for }x\in \S.
}
Define 
\begin{align*}
    Z=\left\{(x,t)\in\S\times\mbR:0<t\leq\frac{1}{\max{\kappa^F_i(x)}}\right\},
   \end{align*} 
   and
   \begin{align}
    &\zeta_F: Z\to \mbR^{n+1},\\
    &\zeta_F(x,t)=x-t\nu_F(x).\label{parallel}
\end{align}

{\bf Claim}. $\Om\subset\zeta_F(Z)$.

Recall that $\mcW_r(x_0)$ is the Wulff shape centered at $x_0$ with radius $r$. 
For any $y\in\Om$, we consider a family of Wulff shapes $\left\{\mcW_r(y)\right\}_{r\geq0}$.
Since $y\in\Om$ is an interior point,  we definitely have $\mcW_r(y)\subset\Om$ for $r$ small enough, and hence as $r$ increases, there exists $x_y\in\S$ and $r_y>0$, such that $\mcW_{r_y}(y)$ 
touches $\S$ for the first time from the interior at  $x_y\in\S$.

{\bf Case 1}. $x_y\in\mathring{\S}$.

In this case, 
the Wulff shape $\mcW_{r_y}(y)$ is tangent to $\S$ at $x_y$, and hence
\eq{\label{eq-nu}
\nu(x_y)=\nu^{\mcW}(x_y),
} 
where $\nu^{\mcW}$  denotes the outward unit normal  of $\mcW_{r_y}(y)$.  
Let $\nu_F^\mcW(x_y)$ be the outward anisotropic normal to $\mcW_{r_y}(y)$ at $x_y$. \eqref{eq-nu} implies 
\eq{\nu_F(x_y)=\nu^\mcW_F(x_y),} which in turn together with \eqref{normal} implies
\eq{\label{case1} y=x_y-r_y \nu ^ {\mcW}_F(x_y) =x_y-r_y \nu _F(x_y)
.}
Moreover, since $\mcW_{r_y}(y)$ touches $\S$ from the interior, we easily infer 
\eq{\label{dnu-eq}
\rd\nu
\leq\rd\nu^\mcW,
} 
in the sense that
the coefficient matrix of the difference of two classical Weingarten operators $\rd\nu-\rd\nu^\mcW$ is semi-negative definite.
It follows from \eqref{eq-nu} and \eqref{dnu-eq} that  \eq{\label{an-dnu-eq}
A_F(\nu)\circ\rd\nu
\leq A_F(\nu^\mcW)\circ\rd\nu^\mcW.
}
Since the  anisotropic principal curvatures of $\mcW_{r_y}(y)$ are equal to $\frac{1}{r_y}$, we see from \eqref{an-dnu-eq} that
\begin{align*}
\max_{1\leq i\leq n}{\kappa^F_{i}(x_y)}\le\frac 1{r_y}.
\end{align*} 
Invoking the definition of $Z$ and $\zeta_F$, we find that $y\in\zeta_F(Z)$ in this case.

{\bf Case 2}. $x_y\in\p\S$.

We will rule out this case thanks to condition \eqref{condi-anisotropic-freebdry}.

Let $\nu_F^\mcW(x_y)$ be the outward anisotropic normal to $\mcW_{r_y}(y)$ at $x_y$ as above.
From \eqref{normal} we know 
\eq{\Phi(\nu^\mcW(x_y))=
\nu_F^\mcW(x_y)
=\frac{x_y-y}{r_y}.
}
Since the Wulff shape $\mcW_{r_y}(y)$ touches $\S$ from the interior, we have
\eq{\label{ineq-case2-1}
\langle\nu(x_y),\bar N(x_y)\rangle
\geq\langle\nu^\mcW(x_y),\bar N(x_y)\rangle.
}
Since $\nu$, $\nu^\mcW$ and $\bar N$ lie on the same two-plane,
we see that
$\nu$ lies actually in the geodesic joining $\nu^\mcW$ and $\bar N$ in $\mbS^n$.
 It follows from the monotonicity principle, Proposition \ref{Prop-angle-compare}
 and 
\eqref{condi-anisotropic-freebdry}
 that 
\eq{
0\geq
\langle\Phi(\nu(x_y)),\bar N(x_y)\rangle
\geq\langle\Phi(\nu^\mcW(x_y)),\bar N(x_y)\rangle=\frac{1}{r_y}\langle x_y-y,\bar N(x_y)\rangle.
}
On the other hand, since $ K$ is convex and $\p\S\subset{\rm Reg}(\p K)$, it follows that $ K$ lies completely on one side of the tangent space $T_{x_y}\p K$. Moreover, since $y\in\Om\subset K$, we infer easily
\eq{
\langle x_y-y,\bar N(x_y)\rangle>0,
}
which contradicts \eqref{ineq-case2-1},
in other words, it can not be that $x_y\in\p\S$.
The {\bf Claim} is thus proved.

By a simple computation, we find
\eq{
\p_t \zeta_F(x, t)&=-\nu_F(x), \\ 
D_{e^F_i} \zeta_F(x, t)&=\left(1-t\kappa^F_i(x)\right)e^F_i(x).
}
A classical computation yields that the tangential Jacobian of $\zeta_F$ along $Z$ at $(x,t)$ is just
$${\rm J}^Z\zeta_F(x,t)
=F(\nu)\prod_{i=1}^n(1-t\kappa^F_i).$$
By virtue of the fact that $\Om\subset\zeta_F(Z)$, the area formula yields
\eq{
\abs{\Om}
\leq\abs{\zeta_F(Z)}
\leq&\int_{\zeta_F(Z)}\mcH^0(\zeta_F^{-1}(y))\rd y
=\int_Z {\rm J}^Z\zeta_F\rd\mcH^{n+1}\\
=&\int_\Sigma\rd A\int_0^{\frac{1}{\max\left\{\kappa^F_i(x)\right\}}}F(\nu)\prod_{i=1}^n(1-t\kappa^F_i(x))\rd t.
}
By the AM-GM inequality,
and the fact that $\max\left\{\kappa^F_i(x)\right\}_{i=1}^n\geq \frac1n H^F(x)$, we obtain
\eq{
\abs{\Om}
\leq&\int_\S\rd A\int_0^{\frac{1}{\max\left\{\kappa^F_i(x)\right\}}}F(\nu)\left(\frac{1}{n}\sum_{i=1}^n\left(1-t\kappa^F_i(x)\right)\right)^n \rd t\notag\\
\leq&\int_\S F(\nu)\rd A\int_0^{\frac{n}{H^F(x)}} \left(1-t\frac{H^F(x)}{n}\right)^n\rd t\notag\\
=&\frac{n}{n+1}\int_\S \frac{F(\nu)}{H^F}\rd A,
}
which gives \eqref{ineq-HK}.

If equality in \eqref{ineq-HK} holds, it is then standard to show that $\S$ must be a part of a Wulff shape.
Since $\S$ is a part of a Wulff shape, we know that the flow $\zeta_F$ indeed maps every $x\in\S$ and $t=\frac{H^F}{n}$ to the center of this Wulff shape, which implies that $\zeta_F(Z)$ is a Wulff sector and hence naturally induces a  cone. 
Moreover, we must have
\eq{
\abs{\Om}=\abs{\zeta_F(Z)},
}
this in turn shows that at least a part of the convex body $ K$ has to coincide exactly with the cone that is determined by $\S$. thus the cone is convex and we infer from $\p\S\subset{\rm Reg}( K)$ that the cone is nowhere singular other than its vertex. 

\end{proof}

\subsection{Proof of Theorem \ref{Thm-HK-Divisor}}

In the case that $ K$ is a wedge, Theorem \ref{Thm-HK-Divisor} is proved in \cite[Theorem 1.5]{JWXZ22}.
Now we  modify that proof to show Theorem \ref{Thm-HK-Divisor}.
\begin{proof}[Proof of Theorem \ref{Thm-HK-Divisor}]
We begin with the notification that similar as \cite[Proof of Theorem 1.5]{JWXZ22}, we shall follow the proof of Theorem \ref{Thm-HK} with a careful investigation on the case when the first touching point of $S_{r_y}(y)$ with $\S$ occurs at $x_y\in\p\S\cap{\rm Sing}_0( K)$ for some $r_y>0$, since if $x_y\in\p\S\cap{\rm Reg}( K)$ is the case, then the proof is easily completed by letting $F(\xi)=\abs{\xi}$ in Theorem \ref{Thm-HK}.
Throughout the proof, $S_{r_y}(y)$ denotes the sphere with radius $r_y$, centered at $y$ and $B_{r_y}(y)$ denotes correspondingly the open ball.

For simplicity we omit the dependence of $x_y,r_y$ on $y$.
Let $\nu_{B_{r}}$ denote the outward unit normal of $B_r(y)$ at $x\in S_r(y)$ and  $T_x\S$  the tangent space of $\S$ at $x$.
Up to relabelling, we may assume that $x\in{\rm int}(F_1\cap F_2)$.
Note that the blow-up limit of $\p K$ at $x$ is given by two non-opposite $n$-dimensional half-planes $P_1,P_2$, with outer unit normals denoted by $\bar N_1,\bar N_2$, respectively.
The intersection $P_1\cap P_2$ is then an $(n-1)$-dimensional plane.
Moreover, thanks to \cite[Remark 4.1]{JWXZ22}, we may reduce the problem to the $3$-dimensional case, that is, $P_1$, $P_2$ are $2$-dimensional half-planes in $\mbR^3$.

For the case when $\bar N_i\parallel\nu_{B_r}$ for some $i$, and the case when $P_1\cap P_2\subset T_x\S$, we could follow \cite[Proof of Theorem 1.5, Cases 1, 2.1]{JWXZ22} to conclude that $S_r(y)\cap P_i$ is tangent to $\p\S\cap P_i$ at $x$ for some $i$, and hence recovers {\bf Case 2} in the proof of Theorem \ref{Thm-HK}, from which we exclude the possibility of the occurrence of first touching at any such points.

It remains to consider the case when $P_1\cap P_2\nsubset T_x\S$,
which is inspired from the wedge case handled in \cite[Proof of Theorem 1.5, Case 2.2]{JWXZ22}.

Since $P_1\cap P_2\nsubset T_x\S$, it is then essential that $\p\S\cap\p F_i$ is a $1$-dimensional curve locally near $x$.
For simplicity of notations we adopt
\eq{
\p\S^i
=\p\S\cap F_i,\quad
\p B_r^i
=S_r(y)\cap F_i,\quad i=1,2.
}
Let $T_{\p\S^i}$ be the unit tangent vector of $\p\S^i$ at $x$ such that $\langle T_{\p\S^i},\bar N_j\rangle$>0 for $i\neq j$ and $T_{\p B_r^i}$ be the unit tangent vector of $\p B_r^i$ at $x$ such that $\langle T_{\p B_r^i},\bar N_j\rangle>0$ for $i\neq j$.
Since $\nu$ and $\bar N_i$ are perpendicular to $T_{\p\S^i}$, we know that $T_{\p\S^i}$ is parallel to $\nu\wedge\bar N_i$.
From the same reason we deduce that $T_{\p B_r^i}$ is parallel to $\nu_{B_r^i}\wedge\bar N_i$.
Let $l$ denote the unit vector in $P_1\cap P_2$, pointing outwards with respect to $\Om$ at $x$.

At this point we want to exclude the possibility of $P_1\cap P_2\nsubset T_x\S$ by exploiting the fact that $x$ is the first touching point.
More precisely, we may show that
\eq{\label{ineq-JWXZ22-(37)}
\langle T_{\p B_r^i},l\rangle
\geq\langle T_{\p\S^i},l\rangle
}
thanks to the first touching, as done in \cite[(37)]{JWXZ22}.

On the other hand, for $i=1,2$,
let $\eta^i\in(0,\pi)$ be such that $\langle\nu_{B_r},\bar N_i\rangle=-\cos\eta^i$ and $\theta^i\in(0,\pi)$ be such that $\langle\nu,\bar N_i\rangle=-\cos\theta^i$.
It then follows from \eqref{condi-freebdry-Divisor} that
\eq{
0<\theta^i\leq\frac{\pi}{2},\text{ for }i=1,2.
}
For an easier understood clarification, let us currently work in the free boundary situation, that is, $\theta^1=\theta^2=\frac{\pi}{2}$.

Notice that $x\in P_i$ and $y\in{\rm int}( K)$, from the convexity of $ K$ we conclude that
\eq{
-\cos\eta^i
=\langle\nu_{B_r},\bar N_i\rangle
=\frac{1}{r}\langle x-y,\bar N_i\rangle
>0,
}
and hence
\eq{
\eta^i>\frac{\pi}{2}=\theta^i,\text{ for }i=1,2,
}
which leads to a contradiction since \eqref{ineq-JWXZ22-(37)} is equivalent to (see \cite[(38),(39)]{JWXZ22})
\eq{
\cos\eta^2+\cos\eta^1\cos\alpha
\geq0,\\
\cos\eta^1+\cos\eta^2\cos\alpha
\geq0,
}
where $\alpha\in(0,\pi)$ is the dihedral angle of $P_1,P_2$.

The general situation when $0<\theta^i\leq\frac{\pi}{2}$ follows similarly from \cite[Case 2.2]{JWXZ22}, with $\mathbf{k}_0=0$ therein.
Here we point out that in \cite{JWXZ22}, shifting in the direction of the constant vector $\mathbf{k}_0$ preserves the capillarity of the parallel surfaces.
In the free boundary situation, there is no shifting needed.

We have shown that the first touching only occurs at some $x\in\S\setminus\p\S$, the rest of the proof then follows directly from that of Theorem \ref{Thm-HK}.
\end{proof}
\section{Alexandrov-Type Theorems}\label{Sec-4}
In this section we prove Theorems \ref{Thm-Alex} and \ref{Thm-Alex-Divisor}.

\begin{proposition}\label{Prop-Minko}
Let $\mcC$ be a cone with vertex at the origin and  $\S\subset\overline{\mcC}$ be a compact, embedded, anisotropic free boundary $C^2$-hypersurface in $\mcC$ with $\p\S\subset{\rm Reg}(\p\mcC)$.
Then the following Minkowski-type formula holds:
\eq{\label{formu-Minko}
\int_\S nF(\nu)-H^F(x)\langle x,\nu\rangle\rd A=0.
}
The conclusion holds true if
$\mcC\in\mcK_P$ is of polytope-type and $\p\S\subset{\rm Reg}(\p\mcC)\cup{\rm Sing}_0(\p\mcC)$.
\end{proposition}
\begin{proof}
We define a $C^1$-vector field on $\S$:
\eq{
X_F(x)\coloneqq
F(\nu(x))x-\langle x,\nu(x)\rangle\nu_F(x).
}
As observed in \cite{JWXZ23}, $X_F$ is in fact a tangential vector field on $\S$ with
\eq{
{\rm div}_\S(X_F)
=nF(\nu)-H^F\langle x,\nu\rangle,
}
where ${\rm div}_\S$ is the divergence operator on the hypersurface $\S$.
Integrating this over $\S$ and using the divergence theorem, 
we obtain
\eq{
\int_\S nF(\nu)-H^F(x)\langle x,\nu\rangle\rd A
=&\int_{\p\S}F(\nu)\langle x,\mu\rangle-\langle x,\nu\rangle\langle\nu_F,\mu\rangle\rd\mcH^{n-1}\\
=&\int_{\p\S}\langle F(\nu)\mu-\langle\nu_F,\mu\rangle\nu,x\rangle\rd \mcH^{n-1}\\
=&\int_{\p\S}\langle\mu_F,x\rangle\rd \mcH^{n-1}=0,
}
where we have used \eqref{condi-aniso-conormal} in the last equality. This proves \eqref{formu-Minko}.
\end{proof}

\begin{proposition}\label{Prop-elliptic-point}
Let $\mcC$ be a cone with vertex at the origin and  $\S\subset\overline{\mcC}$ be a compact, embedded, anisotropic free boundary $C^2$-hypersurface in $\mcC$ with $\p\S\subset{\rm Reg}(\p\mcC)$.
Then $\S$ has at least one elliptic point, i.e., a point where all the anisotropic principal curvatures are positive.
The conclusion holds true if
$\mcC\in\mcK_P$ is of polytope-type and $\p\S\subset{\rm Reg}(\p\mcC)\cup{\rm Sing}_0(\p\mcC)$.
\end{proposition}
\begin{proof}
Consider the family of  Wulff shapes $\mcW_r(O)$, where $O$ is the origin. 
Since $\S$ is compact, for $r$ large enough, $\S$ lies inside the domain bounded by the Wulff shape $\mcW_r(O)$.
Hence we may find the smallest $r$, say $r_0>0$, such that  $\mcW_{r_0}(O)$ touches $\S$ for the first time at some $x_0\in\S$ from exterior.

If $x_0\in\mathring{\S}$, then $\S$ and $\mcW_{r_0}(O)$ are tangent at $x_0$.
If $x_0\in\p\S\cap{\rm Reg}(\p\mcC)$,
observe that there holds
\eq{\label{eq-nuW-E}
\langle\nu_F^\mcW(x_0),\bar N(x_0)\rangle
=\langle\frac{x_0}{r},\bar N(x_0)\rangle
=0
=\langle\nu_F(x_0),\bar N(x_0)\rangle,
}
thanks to the fact that $\S$ is an anisotropic free boundary  hypersurface, which also shows that $\S$ and $\mcW_{r_0}(O)$ are mutually tangent at $x_0$.
In both cases, we may use a similar argument as in the proof of Theorem \ref{Thm-HK} to conclude that the anisotropic principal curvatures of $\S$ at $x_0$ are larger than or equal to $\frac{1}{r_0}$.

It suffice to consider the case when $\mcC=\Psi(P)$ and $x_0\in\p\S\cap{\rm Sing}_0(\p\mcC)$ in the isotropic case, which follows from a similar argument as \cite[Proposition 5.3]{JWXZ22}.
\end{proof}
\begin{proof}[Proof of Theorem \ref{Thm-Alex}]
From Proposition \ref{Prop-elliptic-point} we know that $H^F$ is a positive constant, and hence the Heintze-Karcher-type inequality \eqref{ineq-HK} holds.

On the other hand, since $H^F$ is a positive constant, we may rearrange the Minkowski-type formula \eqref{formu-Minko} to see that
\eq{
\int_\S\frac{F(\nu)}{H^F}\rd A
=\frac{1}{n}\int_\S\langle x,\nu\rangle\rd A
=\frac{n+1}{n}\abs{\Om},
}
where the last equality holds since $\mcC$ is a cone with vertex at the origin.
This means that the equality case of \eqref{ineq-HK} happens, and hence $\S$ must be an anisotropic free boundary  Wulff cap, 
and $\p\mcC\cap\overline{\Omega}$ is a part of boundary of
some convex cone $\tilde{\mcC}$. 

If $\tilde\mcC=\mcC$, it follows easily that $\S$ is a part of a Wulff shape centered at the origin $O$, with $\p\S\subset{\rm Reg}(\p\mcC)$, which implies that  $\mcC$ is a smooth cone.
If $\tilde\mcC\neq \mcC$, 
let $\tilde{O}$ be the vertex of $\tilde{\mcC}$, we assert that $\tilde O\neq O$, otherwise the cone $\mcC$ can be  determined by $\Sigma$, which leads to a contradiction since the cone $\tilde\mcC$ is also determined by $\Sigma$. From Theorem \ref{Thm-HK}, we have  $\tilde{O}\in\p\mcC\cap\overline{\Omega}$, 
thus it is easy to see that the ray $l=\{t\overrightarrow{O\tilde{O}}:t>0\}$ lies on the boundary of $\mcC$. 
Theorem \ref{Thm-HK} tells us that in the equality case $\tilde O$ is actually an interior point of $\p\mcC\cap\overline{\Omega}$, combining with the fact that $\tilde O\in l\subset \partial \mcC$, we deduce: there exist a constant $\epsilon>0$ such that $(1-\epsilon)\tilde O\in \p\mcC\cap\overline{\Omega}\subset\p\tilde\mcC$ and $(1+\epsilon)\tilde O\in\p\mcC\cap\overline{\Omega}\subset\p\tilde\mcC$; that is to say, $\p\tilde\mcC$ contains a line.  Since $\p\Sigma$ is the boundary of $\p\mcC\cap\overline{\Omega}\subset\p\tilde\mcC$, and $\p\S\subset{\rm Reg}(\p\mcC)$, we see that $\tilde \mcC$ must be a half space, thus there exists a flat portion on $\p\mcC$, and 
$\p\S$ lies on a flat portion of $\p\mcC$.




\end{proof}

\begin{proof}[Proof of Theorem \ref{Thm-Alex-Divisor}]
We first  mention that, thanks to the free boundary condition, for any $x\in\S\cap{\rm Sing}_0(\p\mcC)$, say $x\in{\rm int}(F_i\cap F_j)$, we have $\nu(x)\in T_x(F_i\cap F_j)$. In other words, $\S$ must intersect the edges of $\p\mcC$ transversally, and hence Proposition \ref{Prop-Minko} is applicable.

The proof follows similarly from that of Theorem \ref{Thm-Alex}, for the sake of brevity, here we 
just list the significant changes.
For the case that $\tilde \mcC\neq\mcC$, we can also prove that the cone $\tilde\mcC$ contain a line, combing with the fact that $\tilde \mcC$ is determined by $\Sigma$ and $\p\S\subset{\rm Reg}(\p\mcC)\cup{\rm Sing}_0(\p\mcC)$, we deduce that $\tilde\mcC$ is either  a half space or a wedge. Therefore, for the case that $\tilde \mcC\neq\mcC$, $\p\S$ lies on a flat portion or a wedge portion of $\p\mcC$.

\end{proof}
\bibliographystyle{alpha}
\bibliography{BibTemplate.bib}

\begin{thebibliography}{HLMG09}

\bibitem[Ale62]{Alex62}
A.~D. Aleksandrov.
\newblock Uniqueness theorems for surfaces in the large. {I}.
\newblock {\em Amer. Math. Soc. Transl. (2)}, 21:341--354, 1962.
\newblock doi:
  \href{https://doi.org/10.1090/trans2/021/09}{10.1090/trans2/021/09}.

\bibitem[Bre13]{Brendle13}
Simon Brendle.
\newblock Constant mean curvature surfaces in warped product manifolds.
\newblock {\em Publ. Math. Inst. Hautes \'{E}tudes Sci.}, 117:247--269, 2013.
\newblock doi:
  \href{https://doi.org/10.1007/s10240-012-0047-5}{10.1007/s10240-012-0047-5}.

\bibitem[CP11]{CP11}
Jaigyoung Choe and Sung-Ho Park.
\newblock Capillary surfaces in a convex cone.
\newblock {\em Math. Z.}, 267(3-4):875--886, 2011.
\newblock doi:
  \href{https://doi.org/10.1007/s00209-009-0651-3}{10.1007/s00209-009-0651-3}.

\bibitem[EL22]{EL22}
Nicholas Edelen and Chao Li.
\newblock Regularity of free boundary minimal surfaces in locally polyhedral
  domains.
\newblock {\em Comm. Pure Appl. Math.}, 75(5):970--1031, 2022.
\newblock doi: \href{https://doi.org/10.1002/cpa.22039}{10.1002/cpa.22039}.

\bibitem[Gro14]{Gromov14}
Misha Gromov.
\newblock Dirac and {P}lateau billiards in domains with corners.
\newblock {\em Cent. Eur. J. Math.}, 12(8):1109--1156, 2014.
\newblock
  doi:\href{https://doi.org/10.2478/s11533-013-0399-1}{10.2478/s11533-013-0399-1}.

\bibitem[GX23]{GX23}
Jinyu {Guo} and Chao {Xia}.
\newblock Stable anisotropic capillary hypersurfaces in the half-space, 2023.
\newblock \href{https://arxiv.org/abs/2301.03020}{arXiv:2301.03020}.

\bibitem[HK78]{HK78}
Ernst Heintze and Hermann Karcher.
\newblock A general comparison theorem with applications to volume estimates
  for submanifolds.
\newblock {\em Ann. Sci. \'{E}cole Norm. Sup. (4)}, 11(4):451--470, 1978.
\newblock \href{http://www.numdam.org/item?id=ASENS_1978_4_11_4_451_0}{url}.

\bibitem[HLMG09]{HLMG09}
Yijun He, Haizhong Li, Hui Ma, and Jianquan Ge.
\newblock Compact embedded hypersurfaces with constant higher order anisotropic
  mean curvatures.
\newblock {\em Indiana Univ. Math. J.}, 58(2):853--868, 2009.
\newblock doi:
  \href{https://doi.org/10.1512/iumj.2009.58.3515}{10.1512/iumj.2009.58.3515}.

\bibitem[JWXZ22]{JWXZ22}
Xiaohan {Jia}, Guofang {Wang}, Chao {Xia}, and Xuwen {Zhang}.
\newblock Heintze-karcher inequality and capillary hypersurfaces in a wedge,
  2022.
\newblock \href{https://arxiv.org/abs/2209.13839}{arXiv:2209.13839}.

\bibitem[JWXZ23]{JWXZ23}
Xiaohan Jia, Guofang Wang, Chao Xia, and Xuwen Zhang.
\newblock Alexandrov's theorem for anisotropic capillary hypersurfaces in the
  half-space.
\newblock {\em Arch. Ration. Mech. Anal.}, 247(2):Paper No. 25, 19, 2023.
\newblock doi:
  \href{https://doi.org/10.1007/s00205-023-01861-0}{10.1007/s00205-023-01861-0}.

\bibitem[Li20]{Li21}
Chao Li.
\newblock A polyhedron comparison theorem for 3-manifolds with positive scalar
  curvature.
\newblock {\em Invent. Math.}, 219(1):1--37, 2020.
\newblock
  doi:\href{https://doi.org/10.1007/s00222-019-00895-0}{10.1007/s00222-019-00895-0}.

\bibitem[LX19]{LX19}
Junfang Li and Chao Xia.
\newblock An integral formula and its applications on sub-static manifolds.
\newblock {\em J. Differential Geom.}, 113(3):493--518, 2019.
\newblock doi:
  \href{https://doi.org/10.4310/jdg/1573786972}{10.4310/jdg/1573786972}.

\bibitem[LXZ23]{LXZ23}
Zheng Lu, Chao Xia, and Xuwen Zhang.
\newblock Capillary {S}chwarz symmetrization in the half-space.
\newblock {\em Adv. Nonlinear Stud.}, 23(1):Paper No. 20220078, 14, 2023.
\newblock doi:
  \href{https://doi.org/10.1515/ans-2022-0078}{10.1515/ans-2022-0078}.

\bibitem[PT20]{PT20}
Filomena Pacella and Giulio Tralli.
\newblock Overdetermined problems and constant mean curvature surfaces in
  cones.
\newblock {\em Rev. Mat. Iberoam.}, 36(3):841--867, 2020.
\newblock doi: \href{https://doi.org/10.4171/rmi/1151}{10.4171/rmi/1151}.

\bibitem[Rei77]{Reilly77}
Robert~C. Reilly.
\newblock Applications of the {H}essian operator in a {R}iemannian manifold.
\newblock {\em Indiana Univ. Math. J.}, 26(3):459--472, 1977.
\newblock doi:
  \href{https://doi.org/10.1512/iumj.1977.26.26036}{10.1512/iumj.1977.26.26036}.

\bibitem[Ros87]{Ros87}
Antonio Ros.
\newblock Compact hypersurfaces with constant higher order mean curvatures.
\newblock {\em Rev. Mat. Iberoamericana}, 3(3-4):447--453, 1987.
\newblock doi: \href{https://doi.org/10.4171/RMI/58}{10.4171/RMI/58}.

\bibitem[Ros23]{Rosales23}
C\'{e}sar Rosales.
\newblock Compact anisotropic stable hypersurfaces with free boundary in convex
  solid cones.
\newblock {\em Calc. Var. Partial Differential Equations}, 62(6):Paper No. 185,
  2023.
\newblock doi:
  \href{https://doi.org/10.1007/s00526-023-02528-0}{10.1007/s00526-023-02528-0}.

\bibitem[Wen80]{Wente80}
Henry~C. Wente.
\newblock The symmetry of sessile and pendent drops.
\newblock {\em Pacific J. Math.}, 88(2):387--397, 1980.
\newblock \href{http://projecteuclid.org/euclid.pjm/1102779522}{url}.

\bibitem[WX19]{WX19}
Guofang Wang and Chao Xia.
\newblock Uniqueness of stable capillary hypersurfaces in a ball.
\newblock {\em Math. Ann.}, 374(3-4):1845--1882, 2019.
\newblock
  doi:\href{https://doi.org/10.1007/s00208-019-01845-0}{10.1007/s00208-019-01845-0}.

\bibitem[Zie95]{Gunter95}
G\"{u}nter~M. Ziegler.
\newblock {\em Lectures on polytopes}, volume 152 of {\em Graduate Texts in
  Mathematics}.
\newblock Springer-Verlag, New York, 1995.
\newblock doi:
  \href{https://doi.org/10.1007/978-1-4613-8431-1}{10.1007/978-1-4613-8431-1}.

\end{thebibliography}

\end{document}